\documentclass[12pt,oneside,a4paper]{amsart}

\usepackage[T1]{fontenc}
\usepackage{ae}
\usepackage[english]{babel}
\usepackage{amsmath}
\usepackage{amssymb}
\usepackage{amsthm}
\usepackage{mathrsfs}
\usepackage{enumitem}
\usepackage[foot]{amsaddr}

\usepackage[utf8]{inputenc}  
\usepackage{amscd}
\usepackage{mathrsfs}
\usepackage[all]{xy}
\usepackage[OT1]{fontenc}
\usepackage{amssymb,latexsym}
\usepackage{type1cm,ae}
\usepackage{bbding}
\usepackage[dvipsnames]{xcolor}
\usepackage{pgf,tikz}
\usetikzlibrary[arrows]
\usetikzlibrary{decorations.pathreplacing}

\usepackage{afterpage}
\usepackage[doi=false,isbn=false,url=false,backend=biber,style=alphabetic,maxbibnames=99,minalphanames=5,maxalphanames=10,giveninits=true]{biblatex}

\DeclareFieldFormat{pages}{#1}
\renewbibmacro{in:}{%
	\ifentrytype{article}
	{}
	{\bibstring{in}%
		\printunit{\intitlepunct}}}
\AtEveryBibitem{
	\ifentrytype{book}{%
		\clearfield{pages}%
	}{%
	}%
}

\usepackage{graphicx}
\usepackage{xspace}
\usepackage{setspace}
\usepackage[english]{babel}
\usepackage[dvipsnames]{xcolor}
\usepackage{tikz}
\usepackage{etoolbox}
\usepackage[hidelinks,linktocpage,colorlinks=true,linkcolor=blue,citecolor=red]{hyperref}

\newcommand{\e}{\varepsilon}

\newcommand{\Wreuna}{W^{2-\frac{1}{p},p}(\partial\Om)}
\newcommand{\va}{\enspace}

\newcommand{\Om}{\Omega}
\newcommand{\N}{\mathbb{N}}
\newcommand{\Rnn}{\mathbb{R}^{n}}

\DeclareMathOperator{\spt}{spt}

\DeclareMathOperator{\dist}{dist}

\makeatletter
\renewenvironment{proof}[1][\proofname]{%
   \par\pushQED{\qed}\normalfont%
   \topsep6\p@\@plus6\p@\relax
   \trivlist\item[\hskip\labelsep\bfseries#1\@addpunct{.}]%
   \ignorespaces
}{%
   \popQED\endtrivlist\@endpefalse
}
\makeatother

\newtheorem{thm}{Theorem}[section]

\newtheorem{lemma}[thm]{Lemma}
\newtheorem{cor}[thm]{Corollary}

\theoremstyle{definition}
\newtheorem{rmk}[thm]{Remark}

\numberwithin{equation}{section}
\author{Janne Nurminen}
\address{Department of Mathematics and Statistics, University of Jyväskylä}
\email{janne.s.nurminen@jyu.fi}
\begin{document}

\title[Determining an unbounded potential]{Determining an unbounded potential for an elliptic equation with a power type nonlinearity}

	\begin{abstract}
	In this article we focus on inverse problems for a semilinear elliptic equation. We show that a potential $q$ in $L^{n/2+\e}$, $\e>0$, can be determined from the full and partial Dirichlet-to-Neumann map. This extends the results from \cite{MR4269409} where this is shown for Hölder continuous potentials. Also we show that when the Dirichlet-to-Neumann map is restricted to one point on the boundary, it is possible to determine a potential $q$ in $L^{n+\e}$. The authors of \cite{Salo2022} proved this to be true for Hölder continuous potentials.
	
	\vspace{5pt}
\end{abstract}
\maketitle

\tableofcontents

\section{Introduction}

In this paper we consider an inverse problem of determining a potential in $L^{\frac{n}{2}+\e}$, for positive $\e$, from the Dirichlet-to-Neumann (DN) map related to the boundary value problem for a semilinear elliptic equation
	\begin{equation}\label{intro mainbvp}
	\left\{\begin{array}{ll}
		\Delta u+qu^m=0, & \text{in}\,\, \Om \\
		u=f, & \text{on}\,\, \partial\Om,
	\end{array} \right.
\end{equation}
where $m\geq2$, $m\in\N$, and $\Om\subset\Rnn$ open and bounded. This boundary value problem is well posed for $q\in L^{\frac{n}{2}+\e}(\Om)$ and a certain class of boundary values. In fact we show that there is $\delta>0$ such that for all (see \cite{Leoni2017} for Sobolev spaces)
\[f\in U_{\delta}:=\{h\in \Wreuna\colon ||h||_{\Wreuna}<\delta\}\]
there exists a unique small solution $u\in W^{2,p}(\Om)$ with sufficiently small norm. Here and in the rest of this article, we denote $p:=\frac{n}{2}+\e$. Thus the DN map can be defined as
\begin{equation*}
	\Lambda_q\colon U_{\delta}\to W^{1-\frac{1}{p},p}(\partial\Om),\quad f\mapsto\partial_{\nu}u_f|_{\partial\Om}.
\end{equation*}

Our first main result shows that we can determine the potential from the knowledge of the DN map.
\begin{thm}\label{mainthm}
	Let $\Om\subset\Rnn$, $n\geq2$, be a bounded open set with $C^{\infty}$ boundary, $\e>0$ and $q_1,q_2\in L^{\frac{n}{2}+\e}(\Om)$. Let $\Lambda_{q_j}$ be the DN maps associated to the boundary value problems
	\begin{equation}\label{mainbvp}
		\left\{\begin{array}{ll}
			\Delta u+q_ju^m=0, & \text{in}\,\, \Om \\
			u=f, & \text{on}\,\, \partial\Om,
		\end{array} \right.
	\end{equation}
	for $j=1,2,$ and assume that $\Lambda_{q_1}f=\Lambda_{q_2}f$ for all $f\in U_{\delta}$ with $\delta>0$ sufficiently small.
	Then $q_1=q_2$ in $\Om$.
\end{thm}

This result is a special case of Theorem \ref{secondthm} but we give a proof because it is helpful for the other two main theorems of this paper. Also the proof of Theorem \ref{mainthm} gives a reconstruction formula for the potential $q$ via the Fourier transform (see Corollary \ref{corollary reconstruction}).

The proof Theorem \ref{mainthm} is quite similar as in \cite{MR4188325} and it uses the method of higher order linearization first introduced in \cite{MR3802298} and further developed in the works \cite{MR4104456}, \cite{MR4188325}. The key ingredient in this proof is the following integral identity which characterizes the $m$-th order linearization of the DN map $(D^m\Lambda_q)_0$ at $0$ \cite[Proposition $2.2$]{MR4188325}:
\begin{equation}\label{intro int identity}
	\int_{\partial\Om}(D^m\Lambda_{q_1}-D^m\Lambda_{q_2})_0(f_1,\ldots,f_m)f_{m+1}\,dS = -(m!)\int_{\Om}(q_1-q_2)v_{f_1}\cdots v_{f_{m+1}}\,dx.
\end{equation}
Here $v_{f_k}$ are solutions to $\Delta v_{f_k}=0$ with boundary values $v_{f_k}|_{\partial\Om}=f_k$. Using this integral identity together with a result on density of products of solutions eventually gives $q_1=q_2$ in $\Om$.

Theorem \ref{mainthm} has been proved for Hölder continuous potentials in \cite{MR4104456} and \cite{MR4188325} but in this article we give a first result for a less regular potential (at least to the best of our knowledge). The difference is in proving that \eqref{mainbvp} is well-posed when the potential is in $L^{p}(\Om)$ and defining the DN map as a map from $U_{\delta}$ to $W^{1-\frac{1}{p},p}(\partial\Om)$.

In the linear case $(\Delta+q)u=0$, when $n\geq3$, a similar result for $q\in L^{\frac{n}{2}}(\Om)$ has been obtained in the works \cite{Nachman1992}, \cite{Chanillo1990} and in a more general Riemannian manifold setting in \cite{Ferreira2013}, where they used $L^p$ Carleman estimates in their proof. The case $q\in L^{\frac{n}{2}}(\Om)$ is considered optimal in the sense of standard well-posedness theory and for the strong unique continuation principle \cite{Jerison1985}. There are also results when one assumes that $q\in W^{-1,n}(\Om)$, see for example \cite{Haberman2018}. When $n=2$ the lowest regularity for the potential to have uniqueness in the inverse problem, at least to the best of our knowledge, is $L^{\frac{4}{3}}(\Om)$ \cite{Blåsten2020}. The same result is true on compact Riemannian surfaces with smooth boundary \cite{Ma2020}. In dimension two the unique continuation principle holds for potentials in $L^p(\Om)$ where $p>1$ (see for example \cite{Alessandrini2012}, \cite{Amrein81}).

In addition to the full data case, we consider some partial data results for the Schrödinger equation with unbounded potentials. In particular, let $\Gamma$ be an open subset of the boundary $\partial\Om$. Define the partial Dirichlet-to-Neumann map for $f\in U_{\delta}$, $\spt(f)\subset\Gamma$, as 
\[\Lambda_q^{\Gamma}f=\partial_{\nu}u|_{\Gamma}.\] 
Then from the knowledge of this partial DN map it is possible to determine the potential.

\begin{thm}\label{secondthm}
	Let $\Om\subset\Rnn$, $n\geq2$, be a connected open and bounded set with $C^{\infty}$ boundary and let $\Gamma\neq\emptyset$ be an open subset of the boundary $\partial\Om$. Let $\e>0$, $q_1,q_2\in L^{\frac{n}{2}+\e}(\Om)$ and $\Lambda_{q_j}^{\Gamma}$ be the partial DN maps associated to the boundary value problems
	\begin{equation*}
		\left\{\begin{array}{ll}
			\Delta u+q_ju^m=0, & \text{in}\,\, \Om \\
			u=0, & \text{on}\,\, \partial\Om\setminus\Gamma\\
			u=f, & \text{on}\,\, \Gamma
		\end{array} \right.
	\end{equation*}
	for $j=1,2$. Assume that 
	\begin{equation*}
		\Lambda_{q_1}^{\Gamma}f=\Lambda_{q_2}^{\Gamma}f
	\end{equation*}
	for all $f\in U_{\delta}$ with $\spt(f)\subset\Gamma$, where $\delta>0$ sufficiently small. Then $q_1=q_2$ in $\Om$.
\end{thm}

When the potentials are assumed to be Hölder continuous, then this theorem has been proved in \cite{MR4052205} and \cite{MR4269409} using the method of higher order linearization, which we will also use. Here again the key ingredients are the integral identity \eqref{intro int identity} and a density result for solutions of the Laplacian \cite{Salo2022} (see also \cite[Section 4]{Carstea2021}).

For the linear Schrödinger equation, partial data results with unbounded potentials have been proved only for special cases of partial data. When $n\geq3$, it is proved in \cite{Chung2020} that from the knowledge of the partial DN map in a specific situation it is possible to determine a potential in $L^{\frac{n}{2}}(\Om)$. The authors use a method involving the construction of a Dirichlet Green's function for the conjugated Laplacian. In a similar situation on a manifold setting, \cite{Tzou2018} shows that a potential in $L^{\frac{n}{2}}$ can be determined from a particular case of partial data. When $n=2$ the best known result for the case of an arbitrary open subset of the boundary is for potentials in the Sobolev space $W^{1,p}(\Om)$, for $p>2$ \cite{Imanuvilov2012}.

For partial data results, there is still the case when we are restricted to only one point on the boundary. In the situation of $\Delta u+qu^m$ with the potential $q$ in $C^{\alpha}(\bar{\Om})$ this has been proved in \cite{Salo2022} using the method of higher order linearization. Here we show that the same result holds even if we only assume that $q\in L^{n+\e}(\Om)$ for a positive $\e$.

\begin{thm}\label{thirdthm}
	Let $\Om\subset\Rnn$, $n\geq2$, be a connected open and bounded set with $C^{\infty}$ boundary and let $\Gamma\neq\emptyset$ be an open subset of the boundary $\partial\Om$. Suppose that $\mu\not\equiv0$ is a fixed measure on $\partial\Om$ and let $\e>0$. Assume that $q_1, q_2\in L^{n+\e}(\Om)$ satisfy 
	\begin{equation}\label{measure eq}
		\int_{\partial\Om}\Lambda_{q_1}(f)\,d\mu=\int_{\partial\Om}\Lambda_{q_2}(f)\,d\mu
	\end{equation}
	for all $f\in U_{\delta}$ with $\spt(f)\subset\Gamma$, where $\delta>0$ sufficiently small. Then $q_1=q_2$ in $\Om$. Thus when choosing $\mu=\delta_{x_0}$ for some fixed $x_0\in\partial\Om$ the condition
	\begin{equation*}
		\Lambda_{q_1}(f)(x_0)=\Lambda_{q_2}(f)(x_0)\quad \text{for all}\va f\in U_{\delta}\va\text{with}\va \spt(f)\subset\Gamma
	\end{equation*}
	gives $q_1=q_2$ in $\Om$.
\end{thm}

The proof of this theorem is very similar to the one in \cite{Salo2022} and it uses heavily the identity \eqref{intro int identity} and a density result for solutions of the Laplacian \cite{Salo2022}.

It is an interesting question if in Theorems \ref{mainthm} and \ref{secondthm} it is enough to assume the potential $q$ to be in $L^{\frac{n}{2}}(\Om)$ and if in Theorem \ref{thirdthm} the potential $q$ could be in $L^s(\Om)$ for $s=n$ or even $s<n$. The argument given for Theorems \ref{mainthm} and \ref{secondthm} fails when $q\in L^{\frac{n}{2}}(\Om)$ since the well-posedness (Theorem \ref{wellposed}) relies on Sobolev embedding theorems that fail for the exponent $\frac{n}{2}$. For Theorem \ref{thirdthm} the restriction to $s>n$ comes from Lemma \ref{measure lemma} and that we again use Sobolev embedding theorems that do not work for the exponent $n$ or exponents less than $n$.

The rest of this paper is organized as follows. In section \ref{section well-posedness} we prove the well-posedness of the boundary value problem \eqref{intro mainbvp}. In sections \ref{section proof of main} to \ref{section proof of thirdthm} the proofs for Theorems \ref{mainthm}, \ref{secondthm} and \ref{thirdthm} are given.
\\

\textbf{Acknowledgements.} The author was supported by the Finnish Centre of Excellence in Inverse Modelling and Imaging (Academy of Finland grant 284715).
The author would like to thank the anonymous referee for helpful comments and Mikko Salo for helpful discussions on everything related to inverse problems.

\section{Well-posedness}\label{section well-posedness}

A short reminder for the reader that we denote here and in the rest of this article $p:=\frac{n}{2}+\e$.

\begin{thm}\label{wellposed}(Well-posedness)
	Let $\Om\subset\Rnn, n\geq2,$ be a bounded open set with $C^{\infty}$ boundary, $\e>0$ and let $q\in L^{p}(\Om).$ Then there exist $\delta, C>0$ such that for any 
	\[f\in U_{\delta}:=\{h\in \Wreuna\colon ||h||_{\Wreuna}<\delta\},\]
	there is a unique small solution $u_f$ in the class $\{v\in W^{2,p}(\Om)\colon ||w||_{W^{2,p}(\Om)}\leq C\delta\}$ of the boundary value problem
	\begin{equation}\label{eqschr}
		\left\{\begin{array}{ll}
			\Delta u+qu^m=0, & \text{in}\,\, \Om \\
			u=f, & \text{on}\,\, \partial\Om,
		\end{array} \right.
	\end{equation}
	where $m\in\N$ and $m\geq2$.
	Moreover
	\[||u||_{W^{2,p}(\Om)}\leq C||f||_{\Wreuna},\]
	and there are $C^{\infty}$ maps
	\begin{align*}
		S&\colon U_{\delta}\to W^{2,p}(\Om),\quad f\mapsto u_f,\\
		\Lambda_q&\colon U_{\delta}\to W^{1-\frac{1}{p},p}(\partial\Om),\quad f\mapsto \partial_{\nu}u_f|_{\partial\Om}.
	\end{align*}
\end{thm}

The proof uses the implicit function theorem between Banach spaces \cite[Theorem $10.6$ and Remark $10.5$]{MR2028503} and is very similar to the one in \cite[Proposition $2.1$]{MR4188325}. The difference here is that we replace Hölder spaces with Sobolev spaces and one needs to be careful with various embeddings for these spaces.

\begin{proof}
	Let
	\[X=\Wreuna,\quad Y=W^{2,p}(\Om),\quad Z=L^{p}(\Om)\times\Wreuna\]
	and $F\colon X\times Y\to Z,$
	\[F(f,u)=(Q(u),u|_{\partial\Om}-f),\]
	where $Q(u)=\Delta u+qu^m.$ Let us now show that $F$ has the claimed mapping property. Since $u\in W^{2,p}(\Om)$, this implies that $u|_{\partial\Om}\in\Wreuna$ (see \cite{Leoni2017}) and $\Delta u\in L^{p}(\Om)$. Hence we need to show that the term $qu^m\in L^{p}(\Om)$. Since $2\left(\frac{n}{2}+\e\right)>n$, then by the Sobolev embedding theorem \cite{Leoni2017} $u\in C^{0,\alpha}(\bar{\Om})$, for $0<\alpha<1$, which is a subset of $L^s(\Om)$ for every $1\leq s\leq\infty$. 
%
	Now this implies
	\begin{equation*}
		||qu^m||_{L^{p}(\Om)}\leq||q||_{L^{p}(\Om)}||u^m||_{L^{\infty}(\Om)}\leq||q||_{L^{p}(\Om)}\left(||u||_{L^{\infty}(\Om)}\right)^m<\infty
	\end{equation*}
	and thus $qu^m\in L^{p}(\Om)$. Hence $F$ has the claimed mapping property.
	
	Next we want to show that $F$ is a $C^{\infty}$ mapping. Since $u\mapsto \Delta u$ is a linear map $W^{2,p}(\Om)\to L^{p}(\Om)$, it is enough to show that $u\mapsto qu^m$ is a $C^{\infty}$ map $W^{2,p}(\Om)\to L^{p}(\Om)$. This follows since $u^m$ is a polynomial. More precisely, let $u,v\in W^{2,p}(\Om)$ and use the Taylor formula:
	\begin{align*}
		q(u+v)^m&=\sum_{j=0}^{m}\frac{\partial_u^j(qu^m)}{j!}v^j+\int_{0}^{1}\frac{\partial_u^{m+1}\left(q(u+tv)^m\right)}{m!}v^{m+1}(1-t)\,dt\\
		&=\sum_{j=0}^{m}\frac{\partial_u^j(qu^m)}{j!}v^j.
	\end{align*}
	Now for $||v||_{W^{2,p}(\Om)}\leq1$ the above gives
	\begin{align*}
		\Bigg|\Bigg|q(u+v)^m-\sum_{j=0}^{m}\frac{\partial_u^j(qu^m)}{j!}v^j\Bigg|\Bigg|_{L^{p}(\Om)}=0\leq||v||_{W^{2,p}(\Om)}^{k+1}
	\end{align*}
	and thus the map $u\mapsto q(x)u^m$ is $C^{k}$ (in the sense of \cite[Definition $10.2$]{MR2028503}) for all $k\in\N$. Hence it is a $C^{\infty}$ map and $F$ is also $C^{\infty}$ .
	
	Our aim is to use the implicit function theorem for Banach spaces to get a unique solution for the boundary value problem \eqref{eqschr}. Firstly, the linearization of $F$ at $(0,0)$ in the second variable is
	\begin{equation*}
		D_uF|_{(0,0)}(v)=(\Delta v,v|_{\partial\Om}),
	\end{equation*}
	which is linear and also $F(0,0)=0$. Secondly, $D_uF|_{(0,0)}\colon Y\to Z$ is a homeomorphism. To see this, let $(\phi,g)\in Z$ and consider the boundary value problem
	\begin{equation*}
		\left\{\begin{array}{ll}
			\Delta v=\phi, & \text{in}\,\, \Om \\
			v=g, & \text{on}\,\, \partial\Om.
		\end{array} \right.
	\end{equation*}
	This problem has a unique solution for each pair $(\phi, g)$ (see for example \cite[Theorem $9.15$]{GilbargTrudinger}), and thus $D_uF|_{(0,0)}$ is bijective. We also have the estimate
	\begin{equation*}
		||D_uF|_{(0,0)}(v)||^2_Z=||\Delta v||^2_{L^{p}(\Om)}+||v|_{\partial\Om}||_{\Wreuna}^2\leq M||v||^2_{W^{2,p}(\Om)},
	\end{equation*}
	because the trace operator from $W^{2,p}(\Om)$ to $\Wreuna$ is bounded (see \cite{Leoni2017}). Hence $D_uF|_{(0,0)}$ is also bounded and then the open mapping theorem (see  e.g. \cite[Theorem $8.33$]{MR2028503}) tells us that it is also a homeomorphism. 
	
	Now by the implicit function theorem \cite[Theorem $10.6$]{MR2028503} there exists $\delta>0$, a neighborhood $U_{\delta}=B(0,\delta)\subset X$ and a $C^{\infty}$ map $S\colon U_{\delta}\to Y$ such that $F(f,S(f))=0$ for $||f||_{\Wreuna}\leq\delta$. Now $S$ is also Lipschitz continuous, $S(0)=0, S(f)=u$ and thus we have
	\begin{equation*}
		||u||_{W^{2,p}(\Om)}\leq C||f||_{\Wreuna}
	\end{equation*}
	for $C>0$.
	By redefining $\delta$ if necessary we have the estimates $||f||_{\Wreuna}\leq\delta$,  $||u||_{W^{2,p}(\Om)}\leq C\delta$ and the implicit function theorem gives that $u$ is the unique small solution of $F(f,u)=0$. Also the solution operator $S\colon U_{\delta}\to W^{2,p}(\Om)$ is a $C^{\infty}$ map. Because $u\in W^{2,p}(\Om)$, then $\nabla u\in W^{1,p}(\Om)$. The trace operator is a bounded linear map from $W^{1,p}(\Om)$ to $W^{1-\frac{1}{p},p}(\partial\Om)$ (see \cite{Leoni2017}) and thus $\partial_{\nu}u\in W^{1-\frac{1}{p},p}(\partial\Om)$ is defined almost everywhere on $\partial\Om$. Hence $\Lambda_q$ is a well defined $C^{\infty}$ map between $U_{\delta}$ and $W^{1-\frac{1}{p},p}(\partial\Om)$.
\end{proof}

\begin{rmk}
	In the previous proof, we showed that the mapping $D_uF|_{(0,0)}$ is bijective and bounded and deduced that it is a homeomorphism. An alternative way to see this is to look at the inverse map $(D_uF|_{(0,0)})^{-1}\colon Z \to Y$ and show that it is bijective and bounded. In order to do this, one needs to prove the following estimate:
	\begin{equation*}
		||v||_{W^{2,p}(\Om)}\leq C\left(||\phi||_{L^p(\Om)}+||g||_{W^{2-\frac{1}{p},p}(\partial\Om)}\right),
	\end{equation*}
	where $C>0$ does not depend on $v$, $\phi$ and $g$. This can be done for example by combining the estimate
	\begin{equation*}
		||v||_{W^{2,p}(\Om)}\leq C\left(||\phi||_{L^p(\Om)}+||g||_{W^{2-\frac{1}{p},p}(\partial\Om)}+||v||_{L^p(\Om)}\right)
	\end{equation*}
	from \cite[Theorem $9.1.3$]{Wu2006} with the assumption that $0$ is not a Dirichlet eigenvalue and using a compactness argument.
\end{rmk}

\section{Proof of Theorem \ref{mainthm}}\label{section proof of main}

Using the method of higher order linearization we prove that it is possible to determine a potential in $L^{p}(\Om)$ from the knowledge of full DN map.

\begin{proof}[Proof of Theorem \ref{mainthm}]
	Let $\lambda_1,\ldots,\lambda_m$ be sufficiently small numbers, $\lambda=(\lambda_1,\ldots,\lambda_m)$ and $f_1,\ldots,f_m\in \Wreuna$. Let $u_j(x,\lambda)\in W^{2,p}(\Om)$ be the unique small solution to
	\begin{equation}\label{mainproofbvp}
		\left\{\begin{array}{ll}
			\Delta u_j+q_ju_j^m=0, & \text{in}\,\, \Om \\
			u_j=\sum_{k=1}^{m}\lambda_kf_k, & \text{on}\,\, \partial\Om.
		\end{array} \right.
	\end{equation}
	Differentiating this with respect to $\lambda_l, l\in\{1,\ldots,m\}$ (possible by Theorem \ref{wellposed} which shows that $S$ is a $C^{\infty}$ map) and setting $\lambda=0$ gives that $v_j^l:=\partial_{\lambda_l}u_j(x,\lambda)|_{\lambda=0}$ satisfies
	\begin{equation}\label{laplace}
		\left\{\begin{array}{ll}
			\Delta v_j^l=0, & \text{in}\,\, \Om \\
			v_j^l=f_l, & \text{on}\,\, \partial\Om.
		\end{array} \right.
	\end{equation}
	This has a unique solution in $W^{2,p}(\Om)$ (see for example \cite[Theorem $9.15$]{GilbargTrudinger}) and thus we can define $v^l:=v_1^l=v_2^l.$ Also the first linearizations of the DN maps $\Lambda_{q_j}$ are the DN maps of the Laplace equation.
	
	Let $1<a\leq m-1$ be an integer and $l_1,\ldots,l_a\in\{1,\ldots,m\}$. Then the $a$-th order linearization of \eqref{mainproofbvp} is
	\begin{equation*}
	\left\{\begin{array}{ll}
		\Delta(\partial_{\lambda_{l_1}}\cdots\partial_{\lambda_{l_a}}u_j(x,\lambda)|_{\lambda=0})=0, & \text{in}\,\, \Om \\
		\partial_{\lambda_{l_1}}\cdots\partial_{\lambda_{l_a}}u_j(x,\lambda)|_{\lambda=0}=0, & \text{on}\,\, \partial\Om.
	\end{array} \right.
	\end{equation*}
	and uniqueness of solutions for the Laplace equation gives that $0$ is the only solution. Thus the $a$-th order linearization of the DN maps $\Lambda_{q_j}$ are equal to $0$.
	
	Moving to the $m$-th order linearization, we apply $\partial_{\lambda_1}\cdots\partial_{\lambda_m}|_{\lambda=0}$ to \eqref{mainproofbvp} which results in the boundary value problem
	\begin{equation}\label{product}
		\left\{\begin{array}{ll}
			\Delta w_j=-m!q_j\prod_{k=1}^{m}v^k, & \text{in}\,\, \Om \\
			w_j=0, & \text{on}\,\, \partial\Om.
		\end{array} \right.
	\end{equation}
	Here $w_j=\partial_{\lambda_1}\cdots\partial_{\lambda_m}u_j(x,\lambda)|_{\lambda=0}$ and the functions $v^k, k\in\{1,\ldots,m\},$ are solutions to equation \eqref{laplace} with corresponding boundary values $f_k$. On the left hand side of \eqref{product} we are only left with a product of functions $v^k$, since after differentiating \eqref{mainproofbvp} $m$ times with respect to $\e$, all other terms involve a positive power of $u_j$. Proposition \ref{wellposed} says that the solution $u_j$ depends smoothly on $\e$ and thus when evaluating at $\e=0$, the function $u_j$ vanishes.
	
	By our assumptions we have that $\Lambda_{q_1}\left(\sum_{k=1}^m\lambda_kf_k\right)=\Lambda_{q_2}\left(\sum_{k=1}^m\lambda_kf_k\right)$ and thus $\partial_{\nu}u_1|_{\partial\Om}=\partial_{\nu}u_2|_{\partial\Om}$. Applying $\partial_{\lambda_1}\cdots\partial_{\lambda_m}|_{\lambda=0}$ to this gives $\partial_{\nu}w_1|_{\partial\Om}=\partial_{\nu}w_2|_{\partial\Om}$. Subtracting \eqref{product} for $j=1,2$ and integrating against $v\equiv1$ (a solution of \eqref{laplace}) over $\Om$ implies
	\begin{equation}\label{int identity}
		\int_{\Om}m!(q_1-q_2)\prod_{k=1}^{m}v^k\,dx=-\int_{\Om}\Delta(w_1-w_2)\,dx=-\int_{\partial\Om}\partial_{\nu}(w_1-w_2)\,dS=0.
	\end{equation}
	Let us now choose $v^1,v^2$ to be the Calderón's exponential solutions \cite{MR590275}
	\begin{equation}\label{exp sol.}
		v^1(x):=e^{(\eta+i\xi)\cdot x},\quad v^2(x):=e^{(-\eta+i\xi)\cdot x},
	\end{equation}
	where $\eta, \xi\in\Rnn$, $\eta\perp\xi$ and $|\eta|=|\xi|$, and $v^k\equiv1$ for $k=3,\ldots,m$. Then we get that the Fourier transform of the difference $q_1-q_2$ at $-2\xi$ vanishes. Thus $q_1=q_2$ since $\xi$ was arbitrary. 
\end{proof}

Notice that this proof gives a reconstruction formula for the potential. In particular, inspecting the last lines after equation \eqref{int identity} we have the following result which reconstructs the potential $q$ via its Fourier transform.

\begin{cor}\label{corollary reconstruction}
	Let $\Om\subset\Rnn$, $n\geq2$, be a bounded open set with $C^{\infty}$ boundary, $\e>0$ and $q\in L^{p}(\Om)$. Let $\Lambda_q$ be the DN map associated to the boundary value problem
	\begin{equation*}
		\left\{\begin{array}{ll}
			\Delta u+qu^m=0, & \text{in}\,\, \Om \\
			u=f, & \text{on}\,\, \partial\Om.
		\end{array} \right.
	\end{equation*}
	Then, denoting $\lambda=(\lambda_1,\ldots,\lambda_m)$,
	\begin{equation*}
		\hat{q}(-2\xi)=-\frac{1}{m!}\int_{\partial\Om}\frac{\partial^m}{\partial\lambda_1\cdots\partial\lambda_m}\big|_{\lambda=0}\Lambda_q\left(\sum_{k=1}^{m}\lambda_kf_k\right)\, dS,
	\end{equation*}
	where $f_1, f_2$ are the boundary values of Calderón's exponential solutions \eqref{exp sol.}, $f_k\equiv1$ for $3\leq k\leq m$ and $\hat{q}$ is the Fourier transform of $q$.
\end{cor}

\section{Proof of Theorem \ref{secondthm}}\label{section proof of secondthm}

We prove the partial data result for determining a potential in $L^{p}(\Om)$ by using higher order linearization. The proof uses similar techniques as in \cite{MR4052205} and \cite{MR4269409}.

\begin{proof}[Proof of Theorem \ref{secondthm}]
	Let $\lambda_1,\ldots,\lambda_m$ be sufficiently small numbers, $\lambda=(\lambda_1,\ldots,\lambda_m)$ and $f_1,\ldots,f_m\in \Wreuna$ with $\spt(f)\subset\Gamma$. Let $u_j(x,\lambda)\in W^{2,p}(\Om)$ be the unique small solution to
	\begin{equation*}
		\left\{\begin{array}{ll}
			\Delta u_j+q_ju_j^m=0, & \text{in}\,\, \Om \\
			u_j=\sum_{k=1}^{m}\lambda_kf_k, & \text{on}\,\, \partial\Om.
		\end{array} \right.
	\end{equation*} 
	The first and $m$-th order linearizations are the same as in the proof of Theorem \ref{mainthm}, with corresponding boundary values. We also define $v^l:=v_1^l=v_2^l$ by uniqueness of solutions to \eqref{laplace}. Let $v^{(0)}$ be the solution to
	\begin{equation*}
		\left\{\begin{array}{ll}
			\Delta v^{(0)}=0, & \text{in}\,\, \Om \\
			v^{(0)}=0, & \text{on}\,\, \partial\Om\setminus\Gamma\\
			v^{(0)}=g, & \text{on}\,\, \Gamma,
		\end{array} \right.
	\end{equation*}
	where $g\in C^{\infty}_c(\Gamma)$ with $g$ non-negative and not identically zero. By the maximum principle, $v^{(0)}>0$ in $\Om$. Then subtracting \eqref{product} for $j=1,2$ and integrating against $v^{(0)}$ gives
	the following integral identity (compare to \eqref{int identity})
	\begin{align}\label{int identity 2}
		-\int_{\Om}m!(q_1-q_2)v^{(0)}\prod_{k=1}^{m}v^k\,dx&=\int_{\Om}\Delta(w_1-w_2)v^{(0)}\,dx\\\notag
		&=\int_{\Om}(w_1-w_2)\Delta v^{(0)}\,dx \\\notag
		&+ \int_{\partial\Om}v^{(0)}\partial_{\nu}(w_1-w_2)-(w_1-w_2)\partial_{\nu}v^{(0)}\,dS\\\notag
		&=\int_{\partial\Om}v^{(0)}\partial_{\nu}(w_1-w_2)-(w_1-w_2)\partial_{\nu}v^{(0)}\,dS
	\end{align}
	Here Green's formula and the fact that $\Delta v^{(0)}=0$ in $\Om$ were used.
	Now our assumption on the DN maps coinciding gives $\partial_{\nu}u_1|_{\Gamma}=\partial_{\nu}u_2|_{\Gamma}$ and when applying $\partial_{\lambda_1}\cdots\partial_{\lambda_m}|_{\lambda=0}$ to this, we have $\partial_{\nu}w_1|_{\Gamma}=\partial_{\nu}w_2|_{\Gamma}$.  Also $w_1-w_2=0$ on $\partial\Om$ by \eqref{product} and $v^{(0)}=0$ on $\partial\Om\setminus\Gamma$. Using these \eqref{int identity 2} becomes
	\begin{align}\label{int identity 3}
		-\int_{\Om}m!(q_1-q_2)v^{(0)}\prod_{k=1}^{m}v^k\,dx&=\int_{\partial\Om}v^{(0)}\partial_{\nu}(w_1-w_2)-(w_1-w_2)\partial_{\nu}v^{(0)}\,dS\\\notag
		&=\int_{\partial\Om\setminus\Gamma}v^{(0)}\partial_{\nu}(w_1-w_2)\,dS + \int_{\Gamma}v^{(0)}\partial_{\nu}(w_1-w_2)\,dS\\\notag
		&=0.
	\end{align}
	Now we can apply Theorem 1.3 in \cite{Salo2022} (see also \cite[Section 4]{Carstea2021}) which says that the set of products of two harmonic functions that vanish on $\partial\Om\setminus\Gamma$ is dense in $L^1(\Om)$. Thus we can conclude from \eqref{int identity 3} that
	\begin{equation*}
		m!(q_1-q_2)v^{(0)}\prod_{k=3}^{m}v^k=0\quad \text{in}\va \Om.
	\end{equation*}
	Let $f_k\in C_c^{\infty}(\Gamma)$, $f_k$ non-negative and $f_k>0$ somewhere for $k=3,\ldots,m$. Then again the maximum principle gives that $v^k>0$ in $\Om$. Combining this with $v^{(0)}>0$ in $\Om$ then implies $q_1=q_2\quad \text{in}\va \Om.$
\end{proof}

\section{Proof of Theorem \ref{thirdthm}}\label{section proof of thirdthm}

As in \cite{Salo2022}, we need a lemma stating that the solution to the boundary value problem with a finite Borel measure $\mu$ as boundary value is in $L^r(\Om)$ for $1\leq r<\frac{n}{n-1}$. For the lemma, denote by $r'$ the dual exponent of $1\leq r\leq \infty$.

\begin{lemma}\label{measure lemma}
	Let $\Om\subset\Rnn$, $n\geq2$ be a bounded open set with $C^{\infty}$ boundary and $\mu$ a finite complex Borel measure on $\partial\Om$. Then for the function
	\begin{equation}\label{Psi}
		\Psi(x)=\int_{\partial\Om}P(x,y)\,d\mu(y),\quad x\in\Om,
	\end{equation}
	where $P(x,y)$ is the Poisson kernel for $\Delta$ in $\Om$, we have $\Psi\in L^r(\Om)$, $1\leq r<\frac{n}{n-1}$. Additionally $\Psi$ solves the boundary value problem
	\begin{equation*}
		\left\{\begin{array}{ll}
			\Delta \Psi=0, & \text{in}\,\, \Om \\
			\Psi=\mu, & \text{on}\,\, \partial\Om,
		\end{array} \right.
	\end{equation*}
	where $\Psi=\mu$ on $\partial\Om$ means that for any $w\in W^{2,r'}(\Om)$ with $w|_{\partial\Om}=0$, in trace sense, one has
	\begin{equation}\label{boundary value measure meaning}
		\int_{\partial\Om}\partial_{\nu}w\,d\mu=\int_{\Om}(\Delta w)\Psi\,dx.
	\end{equation}
\end{lemma}

Notice that the left hand side of relation \eqref{boundary value measure meaning} is well defined since $\partial_{\nu}w$ is continuous by the Sobolev embedding theorem (see for example \cite{Leoni2017}): The assumption $w\in W^{2,r'}(\Om)$ says that $\nabla w\in W^{1,r'}(\Om)$. This space embeds to $C^{0,1-\frac{n}{r'}}(\bar{\Om})$ if $r'>n$. Notice that $r'>n$ is equivalent with the assumption that $1\leq r<\frac{n}{n-1}$. Also the right hand side of \eqref{boundary value measure meaning} is well defined by the fact that $\Delta w\in L^{r'}(\Om),\Psi\in L^{r}(\Om)$ implies $(\Delta w)\Psi\in L^1(\Om)$.

The proof of this lemma is the same as in \cite[Lemma 2.1.]{Salo2022}. The only difference when compared to the statement in \cite{Salo2022}, is that we assume $w\in W^{2,r'}(\Om)$ instead of $w\in C^2(\bar{\Om})$.

\begin{proof}[Proof of Theorem \ref{thirdthm}]
	As before, we use the method of higher order linearization. Let $\lambda_1,\ldots,\lambda_m$ be sufficiently small numbers, $\lambda=(\lambda_1,\ldots,\lambda_m)$ and $f_1,\ldots,f_m\in \Wreuna$ with $\spt(f)\subset\Gamma$. Let $u_j(x,\lambda)\in W^{2,p}(\Om)$ be the unique small solution to
	\begin{equation*}
		\left\{\begin{array}{ll}
			\Delta u_j+q_ju_j^m=0, & \text{in}\,\, \Om \\
			u_j=\sum_{k=1}^{m}\lambda_kf_k, & \text{on}\,\, \partial\Om.
		\end{array} \right.
	\end{equation*} 
	The first and $m$-th order linearizations are the same as in the proof of Theorem \ref{mainthm}, with corresponding boundary values. We also define $v^l:=v_1^l=v_2^l$ by uniqueness of solutions to \eqref{laplace}.
	
	Let $\e>0$ and $q_1, q_2\in L^{n+\e}(\Om)$ be such that \eqref{measure eq} holds for all $f\in U_{\delta}$, $\spt(f)\subset\Gamma$ with sufficiently small $\delta$. From $\partial_{\lambda_1}\cdots\partial_{\lambda_m}\Lambda_{q_j}(f)=\partial_{\lambda_1}\cdots\partial_{\lambda_m}\partial_{\nu}u_j|_{\partial\Om}=\partial_{\nu}w_j|_{\partial\Om}$, where $w_j$ is the solution to \eqref{product}, and equation \eqref{measure eq} we get that
	\begin{equation*}
		\int_{\partial\Om}(\partial_{\nu}w_1-\partial_{\nu}w_2)\,d\mu=0.
	\end{equation*}
	
	Let $\Psi\in L^{(n+\e)'}(\Om)$ be the function given by \eqref{Psi} which is a solution to 
	\begin{equation*}
		\left\{\begin{array}{ll}
			\Delta \Psi=0, & \text{in}\,\, \Om \\
			\Psi=\mu, & \text{on}\,\, \partial\Om
		\end{array} \right.
	\end{equation*}
	in the sense of Lemma \ref{measure lemma}. Notice that $(n+\e)'<\frac{n}{n-1}$ and $w_j\in W^{2,n+\e}(\Om)$ because $-m!q_j\prod_{k=1}^{m}v^k\in L^{n+\e}(\Om)$ (see for example \cite[Theorem $9.15$]{GilbargTrudinger}). Thus combining \eqref{boundary value measure meaning} and \eqref{product} gives
	\begin{align*}
		0 = \int_{\partial\Om}(\partial_{\nu}w_1-\partial_{\nu}w_2)\,d\mu = \int_{\Om}\Delta(w_1-w_2)\Psi\,dx = -\int_{\Om}m!(q_1-q_2)\prod_{k=1}^{m}v^k\Psi\,dx,
	\end{align*}
	where each $v^k$ is a solution to the Laplace equation with corresponding boundary value $f_k$. Let $f_3,\ldots, f_m\in C^{\infty}(\partial\Om)$ be such that $\spt(f_k)\subset\Gamma$, $f_k\geq0$ and $f_k>0$ somewhere, then by the maximum principle $v^k>0$ in $\Om$. Choosing the boundary values $f_1, f_2\in C^{\infty}(\partial\Om)$, $\spt(f_1), \spt(f_2)\subset\Gamma$ , we get by elliptic regularity that $v^1, v^2$ are smooth and thus we may apply Theorem $1.3$ from \cite{Salo2022} (see also \cite[Section 4]{Carstea2021}) to get
	\begin{equation*}
		m!(q_1-q_2)v_3\cdots v_m\Psi=0\quad\text{a.e. in}\va\Om.
	\end{equation*}
	
	The positivity of $v_3,\ldots,v_m$ implies that $(q_1-q_2)\Psi=0$ a.e. in $\Om$. Now we claim that $\Psi$ cannot vanish in any set $E\subset\Om$ of positive measure. This can be seen as follows: We argue by contradiction  and assume that $\Psi=0$ in $E\subset\Om$ where $E$ has positive measure. Then by a unique continuation principle (see for example \cite{Harrach2017}, $n>2$, and for $n=2$ \cite{Harrach2019}) $\Psi=0$ in $\Om$. From \cite{Krantz2005} there is a constant $c>0$ such that for all $(x,y)\in\Om\times\partial\Om$
	\begin{equation*}
		c\cdot\frac{\dist(x,\partial\Om)}{|x-y|^n}\leq P(x,y).
	\end{equation*}
	In view of the definition of $\Psi$ in \eqref{Psi} this would imply that $\mu\equiv0$ which is a contradiction. 
	Hence we must have that $q_1=q_2$ a.e. in $\Om$.
\end{proof}

\printbibliography
\end{document}